\documentclass[12pt]{amsart}

\usepackage[T1]{fontenc}
\usepackage{lmodern}
\usepackage{array}
\usepackage[hmargin=4cm,vmargin=30mm]{geometry}
\usepackage{amsthm,amssymb,amsmath}
\usepackage{wrapfig,subfig}
\usepackage{txfonts}
\usepackage{tikz}
\usetikzlibrary{snakes}
\usetikzlibrary{arrows}
\usepackage{hyperref}

\usepackage{graphics}

\newtheorem{theorem}{Theorem}
\newtheorem{proposition}{Proposition}

\newtheorem{corollary}{Corollary}
\newtheorem{definition}{Definition}
\newtheorem{remark}{Remark}

\newtheorem*{lemma*}{Lemma}

\def\Z{\mathbb{Z}}

\def\R{\mathbb{R}}
\def\C{\mathbb{C}}

\begin{document}

\title[Loci with fully degenerate Lyapunov spectrum]{Loci in strata of meromorphic differentials with fully degenerate Lyapunov spectrum}
\author{J. Grivaux}

\address{LATP, Universit\'e d'Aix-Marseille, 39 rue F. Joliot-Curie, 13453 Marseille Cedex 20
} 
  
\email{julien.grivaux@univ-amu.fr}

\author{P. Hubert}

\address{LATP, Universit\'e d'Aix-Marseille, 39 rue F. Joliot-Curie, 13453 Marseille Cedex 20}

\email{pascal.hubert@univ-amu.fr}

\subjclass[2010]{Primary: 30F60, 32G15, 32G20; Secondary: 37H15} \keywords{}

\maketitle

\begin{abstract}
We construct explicit closed $\mathrm{GL}(2; \mathbb{R})$-invariant loci in strata of meromorphic differentials of arbitrary large dimension with fully degenerate Lyapunov spectrum. This answers a question of Forni-Matheus-Zorich.
\end{abstract}

\tableofcontents

\section{Introduction}

Lyapunov exponents of the Teichm\"uller flow have been studied a lot since the work of Zorich (\cite{Zo2}, \cite{Zo3}) and Forni \cite{Fo}. Their understanding is important for applications to the dynamics of interval exchange transformations and polygonal billiards. A big breakthrough is the Eskin-Kontsevich-Zorich formula for the sum of positive Lyapunov exponents \cite{grandEKZ}. 
Given a $\mathrm{SL}(2;\mathbb{R})$ invariant suborbifold of a stratum of quadratic differentials,
they relate the sum $\lambda_1 + \dots + \lambda_g$ to the Siegel-Veech constant of the invariant locus\footnote{For quadratic differentials, two bundles can be considered. In this article, we will only be interested in the bundle with fiber $H^1(X,\mathbb{R})$ over a Riemann surface $X$.  The Lyapunov exponents  of this bundle are often denoted by $\lambda_1^+,  \dots,  \lambda_g^+ $.}.
\par \medskip
By  a theorem of Kontsevich and Forni, the sum  $\lambda_1 + \dots + \lambda_g$ is also the integral over the invariant locus of the curvature of the Hodge bundle along Teichm\"uller discs (\cite{Fo}, \cite{grandEKZ}).  
Using this interpretation,  every Lyapunov exponent is computed for cyclic covers of the sphere branched over 4 points
(\cite{petitEKZ}, \cite{FMZ}, see also \cite{BM}, and \cite{W} for abelian covers). For some cyclic covers, Forni-Matheus-Zorich remarked that the sum $\lambda_1 + \dots + \lambda_g$ is equal to zero \cite[Thm. 35]{FMZ}.  This surprizing fact means that the complex structure of the underlying Riemann surface is constant along the Teichm\"uller disc. Forni-Matheus-Zorich ask whether it is possible to construct other invariant loci with this property (see \cite[p. 312]{FMZ}).
The content of this article is to give a simple explanation of the  phenomenon discovered by Forni-Matheus-Zorich. We prove: 
\begin{theorem} \label{thm:yes}
There exist closed $\mathrm{GL}(2;\mathbb{R})$ invariant loci of quadratic differentials of arbitrarily large dimension with zero Lyapunov exponents.
\end{theorem}

This result can be interpreted in the following way: the projection of such a locus  to the moduli space of {\it compact} Riemann surfaces is constant. Remark that the situation for strata of \textit{abelian} differentials is completely different: there are finitely many invariant suborbifolds with fully degenerate Lyapunov spectrum (meaning in this setting that all exponents are zero except $\lambda_1$ which is $1$), and they are arithmetic Teichm\"uller curves (see \cite{M}, \cite{Fo2}, \cite{FMZ} and \cite{A}).
\par \medskip
\textbf{Acknowledgements} We thank John Hubbard, Howard Masur and Christopher Leininger for helpful discussions. We also thank heartily Dmitri Zvonkine for sharing a very valuable idea.
\section{Background material}
\subsection{The Teichm\"uller flow for translation surfaces} $ $
\par
A translation surface is a pair $(X, \omega)$ where $X$ is a compact Riemann surface and $\omega$ is a holomorphic one-form on $X$. If $S(\omega)$ if the set of the zeroes of $\omega$, there exists an open covering of $\tilde{X}=X\setminus S(\omega)$ and holomorphic charts $\varphi_i \colon U_i \rightarrow X$ such that $\varphi_i^* \omega=dz$ for all $i$. For such an atlas, the transition functions are translations.
The form $\omega$ induces a flat metric $|\omega^2|$ on the open surface $\tilde{X}$, whose area is the integral $\frac{i}{2} \int_X \omega \wedge \bar{\omega}$. 
We could have taken meromorphic $1$-forms instead of holomorphic ones, but in that case the area of the surface is never finite. 
\par \medskip
There is a natural action of $\mathrm{GL}(2, \R)$ on translation surfaces given as follows: first we pick an atlas of charts of $\tilde{X}$ where all transitions functions are translations by some complex vectors $v_{ij}$ which we will consider as vectors in $\R^2$. Then, for any $M$ is $\mathrm{GL}(2, \R)$, we get an open surface $\tilde{X}_M$ defined by an atlas whose transition functions are translations by $M. v_{ij}$. This surface is diffeomorphic to $\tilde{X}_M$. Therefore, we can fill the holes and extend the complex structure in a unique way: the result is a compact Riemann surface $X_M$ diffeomorphic to $X$ endowed with a meromorphic differential $\omega_M$ of finite volume, hence holomorphic. The action of $\mathrm{GL}(2; \R)$ is defined by the formula $M. (X, \omega)=(X_M, \omega_M)$. The action of $\mathrm{SL}(2; \R)$ preserves the volume of translation surfaces.
\par \medskip
The subgroup of $\mathrm{SL}(2; \R)$ of matrices $M$ such that $M. (X, \omega)=(X, \omega)$ up to diffeomorphism is called the Veech group of $(X, \omega)$. If $M_t=\begin{pmatrix}
e^t & 0\\
0 & e^{-t} 
\end{pmatrix}$
then the curve $(X_t, \omega_t):=M_t.(X, \omega)$ is called the orbit under the Teichm\"uller flow of $(X, \omega)$. If $M_{\theta}$ is a rotation in $\mathrm{SO}(2, \R)$, then $M_{\theta} .(X, \omega)=(X, e^{i \theta} \omega)$. 
\par \medskip
If we fix multiplicities $(m_1, \ldots, m_r)$ such that $\sum_{i=1}^r m_i=2g-2$, the stratum of translation surfaces $\mathcal{H}(m_1, \ldots, m_r)$ is the set of translations surfaces $(X, \omega)$ where $\omega$ has $r$ distinct zeroes of multiplicities $m_1, \ldots, m_r$ modulo diffeomorphism. The normalized stratum $\mathcal{H}_1(m_1, \ldots, m_r)$ is the locus of flat surfaces with unit area in $\mathcal{H}(m_1, \ldots, m_r)$, and the projective stratum $\mathrm{P} \mathcal{H}(m_1, \ldots, m_r)$ is obtained by taking the quotient of $\mathcal{H}(m_1, \ldots, m_r)$ under the natural $\C^{\times}$-action on forms. Strata and projective strata are complex orbifolds of respective dimensions dimensions $2g+r-1$ and $2g+r-2$ if $g \geq 2$.
\par \medskip
There are standard coordinates on the stratum $\mathcal{H}(m_1, \ldots, m_r)$, called period coordinates. Fix $(X, \omega)$ in this stratum, and let $A_1, \ldots, A_g, \mathrm{B}_1, \ldots, \mathrm{B}_g$ be a symplectic basis of $\mathrm{H}_1(X, \Z)$ and $C_1, \ldots, C_{r-1}$ be $r-1$ paths joining a zero of $\omega$ to all the $r-1$ other zeroes. The map
\[
(X, \omega) \rightarrow \left( \int_{A_1} \omega, \ldots, \int_{A_g} \omega, \int_{\mathrm{B}_1} \omega, \ldots, \int_{\mathrm{B}_g} \omega,  \int_{C_1} \omega, \ldots, \int_{C_{r-1}} \omega \right)
\]
yields an orbifold chart on $\mathcal{H}(m_1, \ldots, m_r)$. These charts allow to define a canonical volume element on $\mathcal{H}(m_1, \ldots, m_r)$, $\mathcal{H}_1(m_1, \ldots, m_r)$, and $\mathrm{P} \mathcal{H} (m_1, \ldots, m_r)$. By classical results of Masur and Veech, connected components of projective strata have finite volume.
\par \medskip
Let $\mathfrak{\mathbb{H}}=\mathrm{SL}(2; \R)/\mathrm{SO}(2)$ denote the Poincar\'e upper-half plane. For any $(X, \omega)$ in a projective stratum, the $\mathrm{SL}(2; \R)$-action factorizes to a holomorphic map 
\[
\mathbb{H} \rightarrow \mathrm{P} \mathcal{H}(m_1, \ldots, m_r)
\]
which is an immersion. The image of this map is called a Teichm\"uller disc, it is stable under the Teichm\"uller flow. Besides, Teichm\"uller discs induce a smooth foliation with holomorphic leaves on $\mathrm{P} \mathcal{H}(m_1, \ldots, m_r)$.
\par \medskip
Assume that the Veech group $\Gamma$ of $(X, \omega)$ is a lattice in $\mathrm{SL}(2; \R)$. Then the image $\mathbb{H} / \Gamma$ of the corresponding Teichm\"uller disc in the projective stratum is called a Teichm\"uller curve.
\par \medskip
All these considerations generalize to the so-called half-translation surfaces, which are pairs $(X, q)$ where $q$ is a quadratic holomorphic (for the time being) differential on $X$. The transitions functions of a well-choosen atlas of charts on the open surface are half translations, that is either translations or flips. The area of the flat metric on $\tilde{X}$ is $\frac{1}{2} \int_X |q|$, and we still have an action of $\mathrm{GL}(2, \R)$ as well as a Teichm\"uller flow. The period coordinates on strata of quadratic differentials are obtained as follows: for any $(X, q)$ in a stratum, we take the twofold branched covering $\mathfrak{p} \colon \widehat{X} \rightarrow X$ given by the holonomy representation of $q$, which is given by a morphism from $\pi_1(X)$ to $\mathbb{Z}/2\mathbb{Z}$. Let $j$ be the corresponding involution acting on $\widehat{X}$. The quadratic differential $\mathfrak{p}^*q$ is the square of an abelian differential $\omega$. The period coordinates of $(X, q)$ are obtained by taking $J$-anti-invariant absolute and relative periods of $(X, q)$.
\par \medskip
However, a major difference happens for quadratic differentials: it is possible to take meromorphic quadratic differentials and still get finite volume for the corresponding flat surface. More precisely, $(X, q)$ has finite volume if and only if $q$ has poles of order at most one. Therefore we have strata, normalized strata and projective strata $\mathcal{Q}(m_1, \ldots, m_r)$,  $\mathcal{Q}_1(m_1, \ldots, m_r)$ and $\mathrm{P}\mathcal{Q}(m_1, \ldots, m_r)$, where $\sum_{i=1}^r m_i=4g-4$ and each $m_i$ is either positive or equal to $-1$.

\par \medskip
Let ${S}$ be a finite subset of $X$ with cardinal $n$, so that $(X, {S})$ gives a point in the marked Teichm\"uller space $\mathcal{T}_{g, n}$ (genus $g$ with $n$ marked points). The cotangent space of $\mathcal{T}_{g, n}$ at $X$ is exactly the space $\mathcal{Q}_S(X)$ of holomorphic quadratic differentials on $X \setminus S$ with poles of order at most one on $S\!$. There is a norm on $\mathcal{Q}_S(X)$ given by $||q||=\int_X |q|$, as well as a dual norm on $\mathcal{Q}_S(X)^*$. The corresponding distance on $\mathcal{T}_{g, n}$ is the Teichm\"uller metric.
\par \medskip
Let us fix $(X, S)$ as well as an element $q$ in $\mathcal{Q}_S(X)$. There is a complex linear form $\mu_q$ on $\mathcal{Q}_S(X)$ given by scalar product with the $\mathrm{L}^{\infty}$ Beltrami differential $\dfrac{|q|}{q}$:
\[
\mu_q(\tilde{q})=\int_X \tilde{q}\,\dfrac{|q|}{q} \cdot
\] 
Note that $\mu_q(q)=\int_X |q| >0$ so that $\mu_q$ is nonzero. Besides, we have $||\mu_q||=1$. The map $q \rightarrow \mu_q$ gives a non-linear isomorphism between the unit spheres of $\mathcal{Q}_S(X)$ and $\mathcal{Q}_S(X)^*$, hence between the unitary cotangent space $\mathrm{U}^* \mathcal{T}_{g, n}$ and the unitary tangent space $\mathrm{U} \mathcal{T}_{g,n}$. 
\par \medskip
If $(X, q)$ is given and $S$ is the set of poles of $q$, the Teichm\"uller flow of $(X, q)$ introduced formerly is the geodesic flow (for the Teichm\"uller metric) on $\mathcal{T}_{g, n}$ starting from $X$ in the direction $\mu_q$.

\subsection{The period mapping} $ $
\par For any compact Riemann surface $X$, $\mathrm{H}^1(X, \C)$ is the orthogonal sum (for the intersection form) of $\Omega(X)$ and $\overline{\Omega(X)}$. Besides, the composition 
\[
\psi: \mathrm{H}^1(X, \R) \rightarrow \mathrm{H}^1(X, \C) \xrightarrow{\mathrm{pr}_1} \Omega(X)
\] 
is an isomorphism. The Hodge norm $|| \, \, ||_{\textrm{Hodge}}$ is the unique norm on $\mathrm{H}^1(X, \R)$ making $\psi$ an isometry.
\par \medskip
Let us now consider a local holomorphic family of curves, that is a proper holomorphic submersion $\pi \colon \mathfrak{X} \rightarrow B$ whose fibers are compact Riemann surfaces of some genus $g$, where $B$ is a small ball in $\C^n$. The Hodge bundle is a holomorphic vector bundle on $B$ of rang $g$ whose fiber at each point $b$ is the vector space $\Omega(X_b)$. The local system $\mathrm{R}^1 \pi_* \,\R_{\mathfrak{X}}$ is trivial, which means that we can canonically identify all the vector spaces $\mathrm{H}^1(\mathfrak{X}_b, \R)$ to some fixed real vector space $\mathbb{V}$ of dimension $2g$. The local period map 
\[
\xi \colon B \rightarrow \mathrm{Gr}(g, \mathbb{V}^{\C})
\] 
associates to any $b$ the subspace $\mathcal{H}_b$ in the Grassmannian of $g$-dimensional complex subspaces of $\mathbb{V}^{\C}$. The derivative of $\xi$ at a point $b$ in $B$ is a linear map from $\mathrm{T}^{1, 0}_b B$ to $\mathrm{End}\,(\mathcal{H}_b, \mathbb{V}^{C}/\mathcal{H}_b)$, which is isomorphic to $\mathrm{End}\,(\mathcal{H}_b, \overline{\mathcal{H}}_b)$.
\par \medskip
The differential of $\xi$ can be explicitely computed: $\xi$ induces a classifying map $\xi_{\mathrm{Teich}} : B \rightarrow \mathcal{T}_g$. Then we have the following formula due to Ahlfors: for any vector $v$ in $\mathrm{T}_b B$ and any elements $\alpha$ and $\beta$ in $\Omega(X_b)$, 
\[
(\beta, \overline{\xi'_v(\alpha)})=\int_X \alpha \otimes \beta . \,\, \xi'_{\mathrm{Teich}}(v).
\]
In this formula, $\xi'_{\mathrm{Teich}}(v)$ is a tangent vector to $\mathcal{T}_g$, hence represented by a Beltrami differential which is a tensor field on $X$ of type $(-1, 1)$. Thus, the integrand in the above formula of type $(2, 0)+(-1, 1)=(1, 1)$. We can also think of $\xi'_{\mathrm{Teich}}(v)$ as a linear form on $\mathcal{Q}(X)$; in this case the above formula reads 
\[
(\beta, \overline{\xi'_v(\alpha)})=\xi'_{\mathrm{Teich}}(v)\,  \{\alpha \otimes \beta\}.
\]

\par \medskip
It is also possible to give another interpretation on $\xi'$. For this we consider the exact sequence of holomorphic vector bundles
\[
0 \rightarrow \mathcal{H} \rightarrow \mathbb{V} \otimes \mathcal{O}_B \rightarrow \overline{\mathcal{H}} \rightarrow 0.
\]
The bundle $\mathbb{V} \otimes \mathcal{O}_B$ carries a natural flat connection (the Gau\ss-Manin connection), but $\mathcal{H}$ is not in general a flat sub-bundle of $\mathbb{V} \otimes \mathcal{O}_B$. A precise way to measure this (see  formula \eqref{courbure} below) is the second fundamental form $\sigma$ associated with this exact sequence and the Gau\ss-Manin connexion; it is a $(1,0)$-form with values in $\mathrm{Hom}\,(\mathcal{H}, \overline{\mathcal{H}})$. A simple calculation shows that 
\begin{equation} \label{comparaison}
\sigma=\xi'.
\end{equation}
\par \medskip
The Hodge bundle $\mathcal{H}$ carries a natural metric given by the intersection form, its curvature form is given by the formula
\begin{equation} \label{courbure}
\Theta_{\mathcal{H}}=\sigma^* \wedge \sigma.
\end{equation}
By "$\wedge$" we mean composition on the fiber and wedge-product on the base. In particular, $i \,\mathrm{Tr} \Theta_{\mathcal{H}}$ is a positive $(1, 1)$-form on $B$.
\par \medskip
For any compact half-translation surface $(X, q)$, Forni's B-form is a bilinear form on $\Omega(X)$ defined by
\[
\mathrm{B}_q(\alpha, \beta)=\int_X \frac{\alpha \otimes \beta}{q}\, |q| \cdot
\]
If $\xi'_{\mathrm{Teich}}(v)$ has unit norm, we can write it as $\mu_q$ for some holomorphic quadratic differential on $X$. Then we have $(\beta, \overline{\xi'_v(\alpha)})=\mathrm{B}_{q}(\alpha, \beta)$. In case of a Teichm\"uller orbit $(X_t, q_t)$, if we differentiate along the vector field $\frac{\partial}{\partial_t}$, we get the formula
\begin{equation} \label{clef}
(\beta, \overline{\xi'_{\partial_t}(\alpha)})=\mathrm{B}_{q_t}(\alpha, \beta).
\end{equation}
Applying Cauchy-Schwarz inequality, $|\mathrm{B}_q(\alpha, \beta)| \leq ||\alpha|| \times ||\beta||$ with equality if and only if there exists a holomorphic one-form $\omega$ and two complex constants $c$ and $c'$ such that $q=\omega^2$, $\alpha=c \omega$ and $\beta=c' \omega$. In particular, if $q$ is meromorphic with simples poles, $||| \mathrm{B}_q ||| <1$.
\par \medskip
We recall now Forni's inequality: let $(X, q)$ be a half-translation surface, $(X_t, q_t)$ be its orbit under the Teichm\"uller flow, $v$ be in $\mathrm{H}^1(X, \R)$ and $t \rightarrow v_t$ be its parallel transport under the Teichm\"uller flow for the Gau\ss-Manin connection. We write $v_t=\chi_t+ \overline{\chi_t}$ where $\chi_t$ is in $\Omega(X_t)$. Then a simple calculation gives
\[
\partial_t \, ||v_t||_{\textrm{Hodge}}=\mathrm{B}_{q_t}(\chi_t, \chi_t).
\]
Combined with the inequality $|||B||| \leq 1$, this gives Forni's inequality
\begin{equation} \label{Forni}
\left|\,\partial_t \,  \{\log ||v_t||_{\,\textrm{Hodge}}\} \, \right| \leq 1.
\end{equation}
\subsection{Lyapunov exponents of the KZ cocycle.}$ $
\par
The parallel transport for the Gau\ss-Manin connection of vectors of $\mathrm{H}^1(X, \R)$ under the Teichm\"uller flow is called the Kontsevich-Zorich cocycle. Recall that the Teichm\"uller flow is ergodic on every connected component $\mathfrak{D}_1$ of the normalized stratum $\mathcal{Q}_1(m_1, \ldots, m_r)$. By Osseledet's theorem, it is possible to associate $2g$ Lyapunov exposants to this cocycle. 
\par \medskip
Forni's inequality (\ref{Forni}) implies that the KZ cocycle is $\log$-integrable, so that the Lyapunov exponents are well-defined. Since the cocycle is symplectic, the Lyapunov spectrum is of the form $\{-\lambda_{1}, -\lambda_{2}, \ldots, -\lambda_{g}, \lambda_{g}, \ldots, \lambda_{g-1}, \ldots, \lambda_{1} \}$ where $\lambda_1 \geq \ldots \geq \lambda_g$.
\par \medskip
Note that the exponents $\lambda_i$ are called $\lambda_i^{+}$ in numerous papers (e.g. in \cite{grandEKZ}, \cite{FMZ}). The exponents $\lambda_i^{-}$ will never be considered in the article.
\par \medskip
By \eqref{Forni}, all $\lambda_i$'s are at most one. If the component $\mathfrak{Q}$ is orientable, which means that every quadratic differential occuring in the stratum is the square of an abelian differential, then the top Lyapunov exponent $\lambda_1$ equals one. If not, the norm of Forni's B form is strictly smaller than one so that $\lambda_1 <1$.
\par \medskip
For any $(X, q)$ in a stratum $\mathrm{P} \mathcal{Q} (m_1, \ldots, m_r)$, the Poincar\'e metric on $\mathbb{H}$ induces a metric on the Teichm\"uller disc passing through $(X, q)$. The corresponding volume element defines a relative $(1,1)$ form $\mathrm{dV}_{\mathrm{Teich}}$, where by "relative" we mean relative with respect to the foliation by Teichm\"uller discs. If $\Theta$ is the curvature of the Hodge bundle on $\mathrm{P} \mathcal{Q} (m_1, \ldots, m_r)$, its trace is also a relative $(1, 1)$ form on the projective stratum. Let $\Lambda \colon \mathrm{P} \mathcal{Q} (m_1, \ldots, m_r) \rightarrow \mathbb{R}$ be defined by the formula 
\[
\Lambda=\frac{\mathrm{Tr}\,\Theta}{\mathrm{dV}_{\mathrm{Teich}}} \cdot
\]
Then Kontsevich-Forni's main formula for the Lyapunov exponents is
\[
\lambda_1 + \ldots + \lambda_g=\int_{\mathfrak{D}} \Lambda(X, q) \, \mathrm{dV}
\]
where $\mathfrak{D}$ is the projection of $\mathfrak{D}_1$ in the projective stratum and 
$\mathrm{dV}$ is the normalized volume element on $\mathfrak{D}$ of total mass one. For any $(X, q)$ in $\mathfrak{D}$, let $\theta_1, \ldots, \theta_g$ be the eigenvalues of Forni's B-form in the direction of the Teichm\"uller flow when diagonalized in an orthonormal basis for the intersection form. Using formul\ae\,  (\ref{courbure}), (\ref{comparaison}) and (\ref{clef}), we see that 
\begin{equation} \label{LF}
\lambda_1 + \ldots + \lambda_g=\int_{\mathfrak{D}} \,\left\{ \theta_1(X, q) + \ldots + \theta_g(X, q) \right\} \, \mathrm{dV}
\end{equation}
Forni's inequality implies that $\theta_i(X, q) \leq 1$ for all $i$ so that $\lambda_1 + \ldots + \lambda_g \leq g$.
\par \medskip
Thanks to the main result of \cite{EM}, any closed $\mathrm{SL}(2; \mathbb{R})$-invariant locus in the projective stratum $\mathrm{P}\mathcal{Q}(m_1, \ldots, m_r)$ is affine in period coordinates, hence carries a natural $\mathrm{SL}(2; \mathbb{R})$-invariant probability measure. It is also possible to define Lyapunov exponents for this measure, and formula \eqref{LF} holds.
\par \medskip
If $(X, q)$ is any half-translation surface, the closure of its $\mathrm{SL}(2; \mathbb{R})$-orbit in the normalized stratum is affine in period coordinates. It follows from \cite{CE} that almost every direction $\theta$, the real Teichm\"uller flow of $(X, e^{i \theta} q)$ is Osseledets-generic for the corresponding natural probability measure. Therefore it makes sense to consider Lyapunov exponents of $(X, q)$, and formula \eqref{LF} is still valid if we integrate on the closure of the $\mathrm{PGL}(2; \mathbb{R})$-orbit.

\section{The determinant locus} 
\subsection{General properties}

Let $\mathfrak{D}$ be a connected component of the projective stratum $\mathrm{P} \mathcal{Q}(m_1, \ldots, m_r)$. 
\begin{definition}
The determinant locus of $\mathfrak{D}$ is the set of elements $(X, q)$ in $\mathfrak{D}$ such that for all holomorphic $1$-forms $\alpha$ and $\beta$ on X$, \mathrm{B}_q(\alpha, \beta)=0$.
\end{definition}
Let us now recall Noether's theorem (see \cite[p. 104 \& 159]{FK}):
\begin{proposition} \label{5hdumat}
Let $X$ be a compact Riemann surface of genus $g$ and 
\[
\tau \colon \mathrm{Sym}^2 \, \Omega^1(X) \rightarrow \mathcal{Q}(X)
\] 
be the multiplication map.
\begin{enumerate}
\item[(i)] If $X$ is not hyperelliptic or if $g \leq 2$, $\tau$ is surjective.
\item[(ii)] If $X$ is hyperelliptic, $\mathrm{Im}\, (\tau)$ has codimension $g-2$ in $\mathcal{Q}(X)$ and consists of the quadratic differentials invariant by the hyperelliptic involution.
\end{enumerate}
\end{proposition}
Since $\tau$ is the transpose of the derivative of the period map, Noether's result has the following geometric interpretation:
\begin{proposition}[Infinitesimal Torelli's theorem, {\cite[Cor. 10.25]{V}}] $ $ \par
Let $\xi \colon \mathcal{T}_g \rightarrow \mathbb{H}_g$ be the period map. Then $\xi$ is an immersion outside the hyperelliptic locus or everywhere if $g \leq 2$, and the restriction of $\xi$ to the hyperelliptic locus is also an immersion.
\end{proposition}

Remark that Forni's $B$-form factors through $\mathrm{Im}\, \tau$, and can be extended naturally to $\mathcal{Q}(X)$ by the formula $\mathrm{B}_q(\tilde{q})=\int_X \tilde{q} \,\dfrac{|q|}{q} \cdot$
\par \medskip
The key proposition of this section is:
\begin{proposition} \label{prop:forget} Let $(X,q)$ be a half-translation surface, $n$ the number of poles of $q$, and $\mathbb{D}$ be its Teich\-m\"uller disc. Then the following are equivalent:
\begin{enumerate} 
\item[(i)] $\mathbb{D}$ lies in the determinant locus.
\item[(ii)] The forgetful map $\mathcal{T}_{g, n} \rightarrow \mathcal{T}_g$ maps $\mathbb{D}$ to a point.
\item[(iii)] For any $(X_t, q_t)$ in $\mathbb{D}$, the extension of $\mathrm{B}_{q_t}$ to $\mathcal{Q}(X_t)$ vani\-shes.
\item[(iv)] All Lyapunov exponents of $(X,q)$ are zero.
\end{enumerate}
\end{proposition}

\begin{proof} $ $
\par
\noindent (i) $\Rightarrow$ (ii) Using \eqref{clef}, the composite map $\mathbb{D} \hookrightarrow \mathcal{T}_{g, n} \rightarrow \mathcal{T}_g \xrightarrow{\tau} \mathbb{H}_g$ has zero derivative. Assume that $\mathbb{D}$ is not contained in the hyperelliptic locus. Thanks to the infinitesimal Torelli theorem, $\mathbb{D}$ is mapped to a point via the forgetful map $\mathcal{T}_{g, n} \rightarrow \mathcal{T}_g$. Assume now that $\mathbb{D}$ is contained in the hyperelliptic locus. Then the restriction of $\tau$ to this locus is again an immersion, and we can apply the same argument.
\par \medskip
\noindent (ii) $\Rightarrow$ (iii) If $(X_t, q_t)$ is a point in $\mathbb{D}$, the derivative of projection of the Teichm\"uller flow of $(X_t, q_t)$ on $\mathcal{T}_{g}$ is the linear form $\tilde{q} \rightarrow \mathrm{B}_{q_t} (\tilde{q})$ on $\mathcal{Q}(X_t)$.
\par \medskip
\noindent (iii) $\Rightarrow$ (i)
Obvious.
\par \medskip
\noindent (i) $\Leftrightarrow$ (iv) Let $\mathcal{V}$ be the closure of the $\mathrm{PSL}(2; \mathbb{R})$-orbit of $X$ and $\nu$ the corresponding $\mathrm{PSL}(2; \mathbb{R})$-invariant probability measure. If $\lambda_1, \ldots, \lambda_g$ are the Lyapunov exponents of $(X, q)$, then 
\[
\lambda_1 + \ldots + \lambda_g=\int_{\mathcal{V}} \left\{ \theta_1(X, q) + \ldots +\theta_g(X, q) \right\} d \nu.
\]
Since all $\theta_i$'s are nonnegative and continuous functions, $\lambda_1= \ldots = \lambda_g=0$ if and only if all $\theta_i$'s vanish on $\mathbb{D}$.

\end{proof}

\begin{corollary} \label{cor:attention}
If $q$ is a holomorphic quadratic differential on $X$, the Teichm\"uller disc of $(X, q)$ is not included in the determinant locus. 
\end{corollary}
\begin{proof}
If $q$ is holomorphic, $\mathrm{B}_q(q)>0$ and we apply Proposition \ref{prop:forget}.
\end{proof}

\begin{remark}
In the hyperelliptic case, it can happen that $q$ is holomorphic but that $(X, q)$ lies in the determinant locus. Let $X$ be an hyperelliptic surface of genus at least $3$, let $j$ be the hyperelliptic involution, and let $q$ be an anti-invariant holomorphic quadratic differential \emph{(}if $X$ is the Riemann surface of a polynomial $w^2-P(z)$, we can take $q=w^{-1}{dz^{\otimes 2}}$\emph{)}. Since any holomorphic $1$-form on $X$ is anti-invariant under $j^*$, $\mathrm{B}_q=0$. Hence $(X, q)$ lies in the determinant locus, but the Teichm\"uller disc of $(X, q)$ goes outside of the hyperelliptic locus. 
\end{remark}
We can give an explicit lower bound on the number $n$.
\begin{proposition} \label{prop:ekzz}
Let $(X, q)$ be a half-translation surface of genus at least $1$ satisfying the equivalent conditions of Proposition \ref{prop:forget}. Then $q$ has at least $\mathrm{max}\,(2g-2, 2)$ poles.
\end{proposition}
\begin{proof}
The fact that the number $n$ of poles of $q$ must be at least one follows  from \cite[Thm 4']{Kra}. To get the lower bound $2g-2$ in the proposition, we use \cite[Thm 2]{grandEKZ} for the closure of the $\mathrm{SL}(2; \mathbb{R})$-orbit $\mathcal{O}$ of $(X, q)$, which is contained in a stratum $\mathrm{P} \mathcal{Q} \bigl( (-1)^n, m_1, \ldots, m_r \bigr) $: we get
\[
\lambda_1 + \ldots + \lambda_g=\dfrac{1}{24} \sum_{j=1}^r \dfrac{m_j \,(m_j+4)}{m_j+2} - \frac{n}{8} + \frac{\pi^2}{3}C_{\mathrm{area}} \Bigl(\overline{O} \Bigr)
\]
where $C_{\mathrm{area}}  \Bigl(\overline{O} \Bigr)$ is a Siegel-Veech constant of the locus $\overline{\mathcal{O}}$ which is nonnegative. Thus, if $\lambda_1 + \ldots + \lambda_g=0$, 
\[
 \sum_{j=1}^r \dfrac{m_j \, (m_j+4)}{m_j+2} \leq 3{n}.
\]
Since $(\sum m_j)-n=4g-4$,
\[
2g-2 \leq \sum_{j=1}^r \dfrac{m_j}{m_j+2} + 2g-2 \leq n
\]
and we get the required estimate.
\begin{remark}
We will see that this bound is asymptotically sharp in \S \ref{winner}.
\end{remark}
\end{proof} 
\subsection{Pillow-tiled surfaces} 
In this section, we give constraints on pillow-tiled surfaces whose Teichm\"uller disc lies in the determinant locus. Let us start with a technical result:
\begin{proposition}
Let $X$ be a Riemann surface of genus $g$, $\mathrm{B}(t_0, \varepsilon)$ a small ball in $\mathbb{C} \setminus \{0, 1, \infty \}$,  and $\varphi \colon X \times \mathrm{B}(t_0, \varepsilon) \rightarrow \mathbb{P}_1$ be a holomorphic map satisfying the following conditions:
\begin{enumerate}
\item For any $t$ in $\mathrm{B}(t_0, \varepsilon)$, $\varphi_t$ is non-constant and $\mathrm{B}(\varphi_t)=\{0, 1, \infty, t\}$. 
\item The configuration of the ramification points of $\varphi_t$ remains constant with $t$.
\end{enumerate}
If $d$ is the degree of the branched coverings $\varphi_t$, then $3 (g-1 )\leq d$.
\end{proposition}

\begin{proof}
For any $x$ in $X$, let $s(x)={\frac{\partial}{\partial t }}_{|t=t_0} \varphi_t(x) \in \mathrm{T}_{\varphi_{t_0}(x)} \mathbb{P}^1 $. Then $s$ is a holomorphic section of the holomorphic line bundle $\varphi_{t_0}^* \mathrm{T} \mathbb{P}^1$. Let $x_0$ be a ramification point of $\varphi_{t_0}$ such that $\varphi_{t_0}(x_0)=0$. Let us assume that $s(x_0) \neq 0$. By the implicit function theorem, the equation $\varphi_t (x)=0$ has a unique solution $(x,t(x))$ depending holomorphically on $x$ for $(x,t)$ near $(x_0, t_0)$. Since $\varphi_{t(x)} (x)=0$, we get 
\[
\partial_t  \varphi (x, t(x)) \,t'(x)+(\varphi_{t(x)})'(x)=0.
\] 
By hypothesis, $x$ is a ramification point of $\varphi_{t(x)}$, i.e. $(\varphi_{t(x)})'(x)=0$. Besides, since $\partial_t  \varphi (t(x), x) \rightarrow s(x_0)$ as $x \rightarrow x_0$, $t'$ vanishes. Hence $\varphi_{t_0}(x)$ vanishes for $x$ near $x_0$, so that $\varphi_{t_0}$ is constant and we get a contradiction. It follows that $s$ vanishes at $x_0$. The same result also holds over any ramification point of $\varphi_{t_0}$ lying over $1$ and $\infty$. Lastly, if $\psi_t(x)=\varphi_t(x)-t$, the argument we used proves that for any ramification point $x$ of $\psi_t$ lying over $0$, ${\frac{\partial}{\partial t }}_{|t=t_0} \psi_t(x)=0$, which means that $s(x)=1$. In particular $s$ in nonzero.
\par \medskip
We can now decompose the ramification divisor $\mathcal{R}$ of the branched covering $\varphi_{t_0}$ as the sum $\mathcal{R}_0+ \mathcal{R}_1+\mathcal{R}_{\infty}+ \mathcal{R}_t$. Besides, we can assume that $\mathrm{deg} \, \mathcal{R}_t$ is smaller than $\mathrm{deg} \, \mathcal{R}_0$, $\mathrm{deg} \, \mathcal{R}_1$ and $\mathrm{deg} \, \mathcal{R}_{\infty}$, otherwise we move the points $0, 1, \infty$ and $t$ by a suitable homographic transformation. Besides, thanks to the Riemann-Hurwitz formula, we have
\[
\mathrm{deg}\,\mathcal{R}=2(g+d-1)
\]
\par \smallskip
\noindent Now $s$ is a nonzero section of the line bundle $\mathcal{L}=\varphi_{t_0}^* \mathrm{T} \mathbb{P}^1 (-\mathcal{R}_{0}-\mathcal{R}_1-\mathcal{R}_{\infty})$, and
\[
0 \leq \mathrm{deg}\, \mathcal{L}=2d- \mathrm{deg}\, \mathcal{R}+ \mathrm{deg}\, \mathcal{R}_t \leq 2d- \frac{3}{4}\,\mathrm{deg}\, \mathcal{R}=\frac{d}{2}-\frac{3g}{2}+\frac{3}{2} \cdot
\]
The result follows.
\end{proof}

\begin{corollary}
Let $(X, q, \pi)$ be a pillow-tiled surface of genus $g$, and let $d$ be the degree of $\pi$. If the Teichm\"uller disc of $(X, q)$ lies in the determinant locus, then $d \geq 3(g-1)$.
\end{corollary}

\begin{remark}
It is not possible to find an upper bound on the primitive degree $d$ in a given connected component of strata since there are infinitely many pillow-tiled surfaces with arbitrary large primitive degree.
\end{remark}
Let $(X, q)$ be a half-translation surface and $(Y, \pi)$ be an arbitrary finite covering of $X$ with branching locus $S\!$.
Assume that for any point $y$ in $Y$ above $S\!$, the ramification index of $\pi$ at $y$ is at least $2$. 
Then $\pi^* q$ is holomorphic, so that $\mathrm{B}_{\pi^*q}$ is non zero on $\mathcal{Q}(Y)$.  Thanks to Corollary \ref{cor:attention}, the Teichm\"uller disc of $(Y, \pi^*q)$ doesn't belong to the determinant locus. Using this observation, we can prove the following: 
\begin{corollary} \label{bien}
Let $(X, q)$ be a half-translation surface and $(Y, \pi)$ be a finite Galois covering of $X$ with branch locus $S\!$. If the Teichm\"uller disc of $(Y, \pi^*q)$ lies in the determinant locus, then at least one pole of $q$ does not belong to $S \!$.
\end{corollary}
As a particular by-product, we get:
\begin{proposition} \label{prop:3points}
Let $(X, q, \pi)$ be a pillow-tiled surface such that $\pi$ is Galois. Then the Teichm\"uller disc of $(X, q)$ lies in the determinant locus if and only the branching locus of $\pi$ contains at most three points.
\end{proposition}

\begin{proof}
Let $q_{\mathrm{st}}$ be the standard meromorphic differential on $\mathbb{P}^1$ with four simple poles such that $q=\pi^* q_{\mathrm{st}}$. Then the branching locus of $\pi$ lies in the set of poles of $q_{\mathrm{st}}$. If $X$ is in the determinant locus, according to Corollary \ref{bien}, one of the poles of $q_{\mathrm{st}}$ is not a branching point of $\pi$. 
\par \medskip
Conversely, assume that the branching locus of $\pi$ has less than four points. If $\{z_1, z_2, z_3, z_4\}$ are the four poles of $q_{\mathrm{st}}$, let us assume that $z_4$ is not a branch point of $\pi$. The complex Teichm\"uller flow of $(\mathbb{P}^1, q_{\mathrm{st}})$ is of the form $(\mathbb{P}^1, q_{t})$ where $q_t$ has poles at $z_1$, $z_2$, $z_3$ and another point $z_4(t)$ such that $[z_1, z_2, z_3, z_4(t)]=t$. Let $\overset{\circ}{X}$ be the \textit{open} Riemann surface obtained by removing $\pi^{-1} \{z_1, z_2, z_3\}$. Then $\overset{\circ}{X}$ is an unramified covering of $\mathbb{P}^1 \setminus \{z_1, z_2, z_3 \}$. It follows that $(\overset{\circ}{X}\setminus \pi^{-1}(z_4(t)), \pi^* q_{t})$ parametrizes the Teichm\"uller disc of $(X, q)$ in $\mathrm{T}^* \mathcal{T}_{g, n}$ (where $n$ is the number of poles of $q$). This disc maps to $\{ X \}$ via the forgetful map $\mathcal{T}_{g, n} \rightarrow \mathcal{T}_{g}$. Thanks to Proposition \ref{prop:forget}, the Teichm\"uller disc of $(X, q)$ lies in the determinant locus.
\end{proof}
Let us now consider pillow-tiled surfaces arising as \textit{cyclic} coverings of the projective line. They are given by a combinatorial datum $(N, a_1, a_2, a_3, a_4)$ where $0 <a_i \leq N$, $\mathrm{gcd}\,(a_1, a_2, a_3, a_4, N)=1$ and $\sum_{i=1}^4 a_i \equiv 0 \, (N)$: the associated cyclic covering is the Riemann surface of the polynomial
\[
w^N-(z-z_1)^{a_1} (z-z_2)^{a_2}(z-z_3)^{a_3}(z-z_4)^{a_4}.
\]
In topological terms, if $(\gamma_i)_{1 \leq i \leq 4}$ are small loops around the $z_i$'s for $1 \leq i \leq 4$, then the kernel of the group morphism
\[
\pi_1(\mathbb{P}^1 \setminus \{z_1, z_2, z_3, z_4\}) \rightarrow \mathbb{Z}/N \mathbb{Z}
\]
given by $\gamma_i \rightarrow a_i$ defines a true cyclic covering of $\mathbb{P}^1 \setminus \{z_1, z_2, z_3, z_4\}$ of degree $N$, which extends to a branched cyclic covering of the projective line.
\par \medskip
In \cite[Thm. 35]{FMZ}, the authors prove that all Lyapunov exponents of the Teichm\"uller curve corresponding to a cyclic covering are $0$ if one of the integers $a_i$ equals $N$. 

\begin{proposition}
If $(X,q)$ is a pillow-tiled surface obtained by a cyclic covering of $\mathbb{P}^1$ with combinatorial datum $(N, a_1, a_2, a_3, a_4)$, then the Teichm\"uller disc of $(X, q)$ lies in the determinant locus if and only if one of the $a_i's$ equals $N$.
\end{proposition}

\begin{proof}
Thanks to Proposition \ref{prop:3points}, it suffices to prove that the projection $\pi$ of the covering is branched at three points or less if and only if one of the $a_i's$ equals $N$. If $\{z_1, z_2, z_3, z_4\}$ are the four points defining the cyclic cover, the ramification index of $\pi$ at any point of $\pi^{-1} (z_i)$ is $\frac{N}{\mathrm{pgcd} \, (N, a_i)}$ qed.
\end{proof}

\subsection{Construction of invariant subvarieties} \label{winner}$ $ \par
In this section, we provide the precise statement underlying Theorem \ref{thm:yes} as well as its proof.
\par \medskip
Let $m_1, \ldots,  m_r$ and $k$ be positive integers such that $(\sum_{i=1}^r m_i)-k=-4$, and let $\mathcal{S}$ be the 
set of couples $(q, x_1, \ldots, x_{k-3})$ such that  such that $q$ is a meromorphic differential on $\mathbb{P}^1$ with simple poles at $0$, $1$ and $\infty$ and the $x_i$'s, and $q$ has $r$ zeroes of order $m_1, \ldots, m_r$. It is a smooth $\mathrm{GL(2; \mathbb{R})}$-invariant submanifold of $\mathrm{T}^* \mathcal{M}_{0, [k]}$ (where the bracket means that the points are ordered).
\par \medskip
Let us fix a covering $(Y, \pi)$ of $\mathbb{P}^1$ ramified over $0$, $1$ and $\infty$, and let $g$ be the genus of $Y$. 
Put 
\[
n=\# \left\{ y \in \,\pi^{-1} \{0, 1, \infty \} \textrm{ such that $\pi$ is unramified at $y$} \right\} + \mathrm{deg}\,(\pi) \times (k-3) \\
\]
We have a natural map
\[
\chi \colon  \mathcal{S} \rightarrow \mathrm{T}^*_{\mathrm{orb}} \mathcal{M}_{g, n}
\]
given by $
\chi(q)=(Y, \pi^* q),
$
where $\mathrm{T}^*_{\mathrm{orb}}$ denotes the \textit{orbifold} cotangent bundle.
\begin{theorem}
Let $\mathcal{W}$ be the image of $\chi$.
\begin{enumerate}
\item The map $\chi \colon \mathcal{S} \rightarrow \mathcal{W}$ is a holomorphic orbifold map, which is a local immersion. Besides, $\mathcal{W}$ is a suborbifold\footnote{By suborbifold, we mean as usually done in this theory "locally finite union of suborbifolds".} of the orbifold cotangent bundle of $\mathcal{M}_{g, n}$ of dimension $r+k-2$.
\par \smallskip
\item $\mathcal{W}$ is $\mathrm{GL}(2; \mathbb{R})$ invariant and lies in the determinant locus, and the projection of $\mathcal{W}$ by the map $\mathrm{T}^*_{\mathrm{orb}} \mathcal{M}_{g, n,} \rightarrow  \mathcal{M}_{g, n} \rightarrow \mathcal{M}_g$ is $\{ Y \}$.
\par \smallskip
\item The Lyapunov spectrum of $\mathcal{W}$ is fully degenerate.
\end{enumerate}
\end{theorem}

\begin{proof}
Let $q$ be a point in $\mathcal{S}$, and $U$ be a small neighborhood of $q$ in $\mathcal{S}$. It is possible to lift locally $\chi$ to a smooth map $\widehat{\chi}$ from $U$ to $\mathrm{T}^* \mathcal{T}_{g, n}$, so that $\chi$ is a smooth orbifold map. 
\par \medskip
If $q_1, q_2$ are two elements in $U$ such that ${\chi}(q_1)={\chi}(q_2)$, then there exists $\varphi$ in $\mathrm{Aut}\, (Y)$ 
such that $\varphi^*(Y, \pi^*q_1)=(Y, \pi^* q_2)$. Thus the fibers of $\widehat{\chi}_{| U}$ are finite. But $\widehat{\chi}$ is affine in period coordinates, so that it is an immersion on $U$. 
\par \medskip
The $\mathrm{GL}(2; \mathbb{R})$-invariance of $\mathcal{W}$ is proved using the same argument as in Proposition \ref{prop:3points}, which corresponds to the particular case $r=0$. 
\par \medskip
Lastly, the fact that the Lyapunov spectrum of $\mathcal{W}$ is totally degenerate results from the implication (ii) $\Rightarrow$ (iv) in Proposition \ref{prop:forget}.

\end{proof}
We end this section by an example proving that the estimate $n \geq 2g-2$ obtained in Proposition \ref{prop:ekzz} is asymptotically sharp. Let $p$ be a prime number, and let $Y$ be a cyclic covering of order $p$ of the projective line fully ramified at three distinct points $z_1$, $z_2$ and $z_3$ obtained by taking the Riemann surface of the polynomial
\[
w^p-(z-z_1)^{a_1} (z-z_2)^{a_2} (z-z_3)^{a_3}
\] where $1 \leq a_1, a_2, a_3 \leq p-1$ and $a_1+a_2+a_3=p$. Then the genus of $Y$ is $\frac{p-1}{2}\cdot$ If $q$ is the pull-back of a meromorphic differential on the sphere with four simple poles at $z_1$, $z_2$, $z_3$ and another point $z_4$, then $\pi^* q$ has exactly $p$ poles, and three zeros of order $p-2$.
Therefore, we have
\[
\underbrace{n}_{p}=\underbrace{2g-2}_{p-3} + \underbrace{\sum_j \frac{m_j}{m_j+2}}_{3-\frac{6}{p}}+ \underbrace{\frac{\pi^2}{3} C_{\mathrm{area}}}_{\frac{6}{p}}
\]

\nocite{*}
\bibliographystyle{alpha}
\bibliography{bib}

\end{document}